\documentclass[10pt,a4paper]{article}
\usepackage[english]{babel}
\usepackage{srcltx}\usepackage{cite}
\usepackage{amsthm}
\usepackage{latexsym,amsfonts,amssymb,amsmath,longtable,stmaryrd,mathrsfs}

\begin{document}

\theoremstyle{plain}
\newtheorem{Lemma}{{ Lemma}}
\newtheorem{Conj}{{Гипотеза}}
\newtheorem{Theo}{{ Theorem}}
\newtheorem{Prob}{{ Problem}}
\newtheorem{Prop}{{Предложение}}

\renewcommand{\proofname}{Proof}

\begin{center}\begin{Large}
\textbf{Intersection of conjugate solvable subgroups in  symmetric groups}\\
Anton Baykalov\\
\end{Large}
\begin{small}
Novosibirsk State University\\
{\it  anton188@bk.ru}
\end{small}
\end{center}

\pagenumbering{arabic}
\begin{abstract}
It is shown that for a solvable subgroup  $G$ of an almost simple group $S$ which socle is isomorphic to $A_n$ $ (n\ge5)$ 
there are  $x,y,z,t \in S$ such that   $G \cap G^x \cap G^y \cap G^z \cap G^t =1.$
\medskip

\noindent{\bf Kay words}: symmetric group, solvable group, almost simple group. 

\end{abstract}

\section*{Introduction}

Assume that a finite group $S$ acts on a set $\Omega.$ An element $x \in \Omega$ is called a $S${\it-regular point} if
$|xS|=|S|$, i.e. if the stabilizer of
$x$
is trivial.
Define the action of the group $S$ on $\Omega^k$ by the rule
$$g: (i_1, \ldots,i_k)\mapsto(i_1g,\ldots,i_kg).$$
If $S$ acts faithfully and transitively on $\Omega$, then the minimal number $k$ such that the set $\Omega^k$ contains a
$S$-regular point is called the
{\it base size} of $S$ and is denoted by  $b(S).$
For a positive integer $m$ the number of $S$-regular orbits on $\Omega^m$ is denoted by $Reg(S,m)$ (this number equals 0 if $m < b(S)$).
If $G$ is a subgroup of $S$ and $S$ acts by the right multiplication on the set $\Omega$ of the right cosets
of $G$ then $S/G_S$ acts faithfully and transitively on $\Omega.$ (Here $G_S=\cap_{g \in S} G^g$ is the core of $G$.) In this case, we denote
 $b(S/G_S)$ and $Reg(S/G_S,m)$ by $b_G(S)$ and $Reg_G(S,m)$ respectively.
Thus $b_G(S)$ is the minimal number $k$ such that there exist  \\ $x_1,\ldots,x_k$ $\in S$ for which $G^{x_1}\cap \ldots \cap H^{x_k}=G_S.$

Consider Problem 17.41 from ``Kourovka notebook''\cite{kt}:
\begin{Prob}
Let $G$ be a solvable subgroup  of a finite group  $S$ and $S$ does not contain nontrivial normal solvable subgroups.
Are there always exist five subgroups conjugated with
$G$ such that their intersection is trivial?
\end{Prob}

The problem is reduced to the case when $S$ is almost simple in \cite{vd}. Specifically, it is proved  that if for each almost simple group $S$ and solvable subgroup  $G$ of $S$ condition $Reg_G(S,5)\ge
5$ holds then for each finite  non\-solvable group $S$  with trivial solvable radical and solvable subgroup  $G$ of $S$ condition $Reg_G(S,5)\ge
5$ holds.

We have proved the following 

\begin{Theo}
Let
$G$
be a solvable subgroup of an almost simple group
$S$
whose socle is
isomorphic to $A_n$, $n\ge5.$
Then $Reg_G(S,5)\ge 5.$ In particular $b_G(S)\le~5.$
\end{Theo}

The proof contains calculations in GAP system. Texts of programs can be found here:\\
{ \small https://goo.gl/rm9l3M}


\section{Preliminary results}

Our totation is standard.\\
By $A \rtimes B$ we denote the semidirect product of groups $A$ and $B$, where $B$ acts on $A$, by
 $A \wr B$  the permutation
 wreath product of groups $A$ and $B$, where $B \le Sym_n$.

If group $G$ acts on the a $\Omega=\Omega_1 \sqcup \Omega_2$ and $\Omega_i$ is an invariant subset,
then denote projection of $G$ on $\Omega_i$ by $G|_{\Omega_i}$. 

\begin{Lemma}\textup{ \cite[Lemma 13]{vd} }\label{EV}
Let $G$ be a finite group with trivial solvable radical, $M$ be a solvable subgroup of  $S_n$, and   $k$ be a natural number such that
for every solvable subgroup $T$ of  $G$ the following inequality holds:
 $Reg_T(G,k)\ge 5$.
 Then the solvable radical of $G \wr M$ is trivial and for every solvable subgroup $S$ of $G \wr M$ the
 following inequality holds: ${Reg_S(G \wr M,k) \ge 5.}$
\end{Lemma}

\begin{Lemma}\textup{ \cite[Theorem 1]{prim} }\label{pr}
Let $H$ be a maximal primitive subgroup of $G=S_n$ or $A_n$; and $H \ne A_n.$
Then  $b_H(G)\le 3 $ for $n\ge 11.$
\end{Lemma}

\begin{Lemma}\label{bs}
Let $H < G$ and $b_H(G) \le 4$. Then $Reg_H(G,5) \ge 5.$
\end{Lemma}
\begin{proof}
If $b_G(S) \le 2$ and $(\omega_1,\omega_2$) is a regular point of $\Omega^2$ (and also $\omega_1 \ne \omega_2$),
then it is easy to see that points
\begin{gather*}
(\omega_1,\omega_1,\omega_1, \omega_1,\omega_2);\\
(\omega_1,\omega_1,\omega_1, \omega_2,\omega_1);\\
(\omega_1,\omega_1,\omega_2, \omega_1,\omega_1);\\
(\omega_1,\omega_2,\omega_1, \omega_1,\omega_1);\\
(\omega_2,\omega_1,\omega_1, \omega_1,\omega_1)
\end{gather*}
are regular in distinct  $S$-orbits.

If $b_G(S)=3$ and $(\omega_1,\omega_2,\omega_3$) is a regular point of $\Omega^3$(and also $\omega_i \ne \omega_j$ for $i \ne j,$  otherwise $b_G(S)<3$),
then it is easy to see that points
\begin{gather*}
(\omega_1,\omega_1,\omega_1, \omega_2,\omega_3);\\
(\omega_1,\omega_1,\omega_2, \omega_3,\omega_1);\\
(\omega_1,\omega_2,\omega_3, \omega_1,\omega_1);\\
(\omega_2,\omega_3,\omega_1, \omega_1,\omega_1);\\
(\omega_1,\omega_1,\omega_2, \omega_1,\omega_3)
\end{gather*}
are regular in distinct  $S$-orbits.

If $b_G(S)=4$ and $(\omega_1,\omega_2,\omega_3, \omega_4$) is a regular point of $\Omega^4$(and also  $\omega_i \ne \omega_j$ for $i \ne j,$  otherwise $b_G(S)<4$),
then it is easy to see that points
\begin{gather*}
(\omega_1,\omega_1,\omega_2, \omega_3,\omega_4);\\
(\omega_1,\omega_2,\omega_3, \omega_4,\omega_1);\\
(\omega_2,\omega_3,\omega_4, \omega_1,\omega_1);\\
(\omega_2,\omega_3,\omega_1, \omega_4,\omega_1);\\
(\omega_2,\omega_1,\omega_3, \omega_1,\omega_4)
\end{gather*}
are regular in distinct  $S$-orbits.
\end{proof}

\begin{Lemma}\label{1}
Let  $G$ be a solvable primitive subgroup of $S$, where  $S$ is isomorphic to   $S_n$ or $A_n$, $n\ge 5$. Then $b_G(S_n)\le 3.$
In particular $Reg_G(S_n,5)\ge 5.$
\end{Lemma}
\begin{proof}
The group $S_n$ is known to possess a  non-trivial solvable primitive subgroup if and only if  $n =p^k$ for a prime $p$ and an integer $k$ \cite[\S4, Theorem 9]{sup}.
The proof is divided into 3 cases:

{\bf Case 1 ($n \ge 11$).}  Subgroup $G$ lies  in a maximal primitive subgroup  $M$ of $S_n$. Then by Lemma \ref{pr} we have 
$b_M(S_n)\le 3$, thus $b_G(S_n)\le 3.$ 

{\bf Case 2 ($n = p$; $p=5$, $p=7$).} 
In view of the known structure of maximal primitive solvable subgroups in  $S_n$ \cite[\S4, Theorem 10]{sup} in this case $G$ is isomorphic to 
$\mathbb{Z}_p \rtimes \mathbb{Z}_{p-1}$. Up to conjugation $G$ is a semidirect product $A \rtimes B$. Here $A=\langle(1,2, \ldots, p) \rangle$, and 
$B$ is the stabilizer of point 1 in $G$. It is easy to verify by hand or using GAP, that in this case $G\cap G^{(1,2)} \cap G^{(1,3)}=1.$	
Furthermore,  $B$ is generated by an odd permutation $b$, i.e. permutations $b(1,2)$  and $b(1,3)$ are even, thus $ G\cap G^{b(1,2)} \cap G^{b(1,3)}=1.$
So, we have $b_G(S) \le 3.$ 

{\bf Case 3 ($n = 8$, $n = 9$).} If $n=9$ then $G$ is isomorphic to $\mathbb{Z}_3^2 \rtimes GL_2(3).$ 
If $n=8$  then $G$ is isomorphic to a subgroup of  $\mathbb{Z}_2^3 \rtimes GL_3(2)$. These two cases can be verified directly by counting orbits using GAP. Note that group $ G $ contains an odd permutation, so it suffices to prove that $b_G(S_n)\le 3.$
\end{proof}

\begin{Lemma}\label{polu}
Let $G$ be a semiregular subgroup of $S_n, n\ge 5$. Then there is $x \in S_n$ such that  $G \cap G^x=1.$
\end{Lemma}
\begin{proof}
It suffices to prove the lemma for the case when the group $ G $ is regular. Indeed, let $G$ be non-regularly and $\{1,2 \ldots , n\}$ is the union of orbits
$\Omega_1, \ldots, \Omega_k$. Then $G$ acts regular on each orbit, so 
$G \le S(\Omega_1) \times \ldots \times S(\Omega_k)$, where 
$S(\Omega_i)$ is the permutation group of $\Omega_i$. Let $G_i$ be the projection of $G$ on $S(\Omega_i)$. Since $G_i$ is regular, there are 
$x_i \in S(\Omega_i)$ such that $G_i \cap G_i^{x_i}=1.$ Thus $G \cap G^{x_1 \cdot \ldots \cdot x_k}=1.$

Let $G$ be regular. If $G$ is elementary  abelian, then by \cite[Theorem 1]{zen} we obtain the required. Consider the case when $G$ is not elementary  abelian.

We show that if $g \in G$ and  $g^{-1} \cdot g^{(1,2)} \ne 1$, then $g^{-1} \cdot g^{(1,2)}$ is not contained in $G$.
Consider possibilities for the structure of the element $g$. Since $G$ is regular,  every non-trivial element of $G$  has no fixed points.  There are three cases:\\
{\bf 1. } Points $1$ and $2$ are contained in distinct independent  cycles, and there are more that two independent cycles, $g=(1,i_1, \ldots, i_{k-1}, i_k)(2,j_1, \ldots, j_l)(...).$ Then 
$g^{-1} \cdot g^{(1,2)}$ fix all points which are not contained in the same cycle as $1$ or $2$, i.e. $g^{-1} \cdot g^{(1,2)} \notin G$.\\
{\bf 2.} Let $g=(1,i_1 \ldots, i_{k-1}, i_k)(2, j_1 \ldots, j_l).$ Since $n \ge 5$, the length of at least one of these cycles is greater than two. Without loss of generality, we can assume that it is the first. Then $g^{-1} \cdot g^{(1,2)}$  fixes $i_k$, i.e. $g^{-1} \cdot g^{(1,2)} \notin G$.\\
{\bf 3.} Let $g=(1, i_1 \ldots i_k, 2, j_1 \ldots, j_{l-1}, j_l)$. Since $n \ge 5$, we have $k\ge2$ or $l\ge2$. Without loss of generality, we can assume that $l\ge2$. Then $g^{-1} \cdot g^{(1,2)}$ fixes $j_l$, i.e. $g^{-1} \cdot g^{(1,2)} \notin G$.

Therefore,  $G \cap G^{(1,2)} \subseteq \{g \in G \mid g^{(1,2)}=g\}$.
Since $G$ is regular, every non-trivial element of $G$ do not have fixed points. 
Thus, if $g^{(1,2)}=g$  then  $\{1,2\}$ is an orbit of $\langle g \rangle$ and $g^2$ fixes points $1$ and $2$, i.e. $g^2=1.$ It means that  $g=(1,2)(i_1,i_2)\ldots (i_{k-1},i_k)$, $n$ is even and $|G \cap G^{(1,2)}|=2.$ In particular, if $n$
odd then $G \cap G^{(1,2)}=1$. 

If $n$ is even then there is an element $g_1$ of order greater than two in $G$ because $G$ is not elementary abelian. Up to conjugation in $ S_n $ we can assume that $g_1$ maps the point 1 to point 2, then,
due to the regularity of $G$, $g_1$ is the only element which maps 1 to 2. Thus $G \cap G^{(1,2)}=1$, because  $(1,2)(i_1,i_2)\ldots (i_{k-1},i_k)$
is not contained in $G$.
\end{proof}

\section{Case of transitive solvable subgroup}

\begin{Lemma}\label{tr}
Let $G$ be a solvable transitive subgroup of $S$, where $S$ is isomorphic to   $S_n$ or $A_n$, $n\ge 5$. Then $Reg_G(S,5)\ge 5.$
\end{Lemma}
\begin{proof}
If $G$ is primitive then the lemma follows from Lemma  \ref{1}. Let $G$ be imprimitive.
Then  $G$ is contained in  $S_k \wr S_l$, here $n=k \cdot l,$
and $\{1,2, \ldots, n\}$ is partitioned into blocks $\Omega_1 \cup \ldots \cup \Omega_l$.
We denote
$G_i=\{g \in G \mid \Omega_ig=\Omega_i\},$ $K_i=G_i|_{\Omega_i}$.  Let $L$ be  the image of
$G$ in $S_l$. Then $G \le (K_1 \times \ldots \times K_l) \rtimes L$ and group
$(K_1 \times \ldots \times K_l) \rtimes L$ is solvable. Thus
we may assume that $G=(K_1 \times \ldots \times K_l) \rtimes L.$

{\bf Case $k \ge 5.$}

If $S=S_n$ then by induction for $K_i$ we have   $Reg_{K_i}(S_k,5)\ge 5,$ so by Lemma \ref{EV} we get $Reg_G(S_k \wr L,5)\ge 5$, whence $Reg_G(S_n,5)\ge 5$.

Assume that $S=A_n.$ If $K_i \le A_k$, then $G \le A_k \wr L$ so,  by Lemma  \ref{EV}, $Reg_G(A_n,5) \ge 5.$
If  $K_i$ is not contained in $A_k$, then there exist an odd permutation in $K_i$ and hence in $G.$
In this case $Reg_G(A_n,5) \ge 5,$ by the fact that $Reg_G(S_n,5)\ge 5$ and $G$ is normalized by an odd permutation.

Now assume that $ k $ equals $ 2$, $3 $, or $ 4 $, and $ S $ equals $ S_n.$
Since $S_n$ is solvable for $n\le4$, we may assume that
$G=(K_1 \times \ldots \times K_l) \rtimes L$, where each $K_i$ is isomorphic to $S_k.$
In particular, there is an odd permutation in group $S_n$ normalizing $G$. Hence we may assume that  $S=S_n.$
Moreover, up to conjugation in $S_n$, $K_i$ acts on $\{k\cdot (l-1)+1, \ldots, k\cdot l \}.$

{\bf Case $k=2$.}
Clearly, the intersection
$((K_1 \times \ldots \times K_l) \rtimes L) \cap ((K_1 \times \ldots \times K_l) \rtimes L)^{(1,2, \ldots, n)}$ 
stabilizes the partitions $\{1,2\}\{3,4\} \ldots \{n-1,n\}$ and $\{2,3\}\{4,5\} \ldots \{n,1\} $. If an element $g$ from the intersection
stabilizes a point (Suppose, for definiteness, that $ g $ stabilizes point 1), then $g$ fixes point 2,
because it is contained  in the same block in the first partition.
Thus, because  the element $g$ fixes point 2, it also fixes point 3, due to the fact that 2 and 3 are contained
in the same block in the second partition, etc. So if $g$ fixes the first element in a block of the first partition then it also stabilizes the second point of the block and the first point of the next block
of the same partition. Thus
$g$ stabilizes all points, hence
$g=1,$ i.e. the intersection
$$((K_1 \times \ldots \times K_l) \rtimes L) \cap ((K_1 \times \ldots \times K_l) \rtimes L)^{(1,2, \ldots, n)}$$
is semiregular
and by Lemma  \ref{polu} we have  $b_G(S_n)\le 4$. So  $Reg_G(S_n,5) \ge 5$ by Lemma \ref{bs}.

{\bf Case $k=3$.} Consider the group 
$$ ((K_1 \times \ldots \times K_l) \rtimes L) \cap ((K_1 \times \ldots \times K_l) \rtimes L)^{(1,2, \ldots, n)} 
\cap ((K_1 \times \ldots \times K_l) \rtimes L)^{(1,2, \ldots, n)^2},$$ 
which stabilizes the partitions 
\begin{gather*}
\{1,2,3\}\{4,5,6\} \ldots \{n-2,n-1,n\};\\ 
\{2,3,4\}\{5,6,7\} \ldots \{n-1,n,1\}; \\
\{3,4,5\}\{6,7,8\} \ldots \{n,1,2\}.
\end{gather*}

If an element $g$ from this group stabilizes a point (for a definiteness, let $g$ stabilizes the point 1) then the point  2 can be moved under the action of $ g $ only to  2 or 3, because of the first partition.  However, points $n$, 1 and 2 compose
a block in the third partition, therefore 2 is fixed by $g$, so  the point  3 is also fixed by $g$.
Since  points 2, 3 and 4 compose
a block in the second partition, the element $g$ has to stabilize the point 4. Thus, repeating the preceding argument, we find that
element $ g $ stabilizes all the points, i.e $g=1,$ which implies that the considered group is semiregular.

Let us find the number of orbits of the group $$R= ((K_1 \times \ldots \times K_l) \rtimes S_l) \cap ((K_1 \times \ldots \times K_l) \rtimes S_l)^{(1,2, \ldots, n)} 
\cap ((K_1 \times \ldots \times K_l) \rtimes S_l)^{(1,2, \ldots, n)^2}.$$

Note that the points 1, 2 and 3 lie in distinct orbits. Indeed, let  $g \in R$ maps 1 to 2. In this case, points of the first block of the first partition stay in this block, i.e.
 $g$ maps 3 to 3, or to 1. In the first case 2  moves to 1, because $g$ stabilizes the first partition, but then $ g $ does not stabilize the second partition. In the second case $ g $ does not stabilize the third partition. Similar arguments show that $ g $ can not map  1 to  3. Due to the symmetry of entries 1, 2 and 3 in the above considerations, these points lie in distinct orbits. Thus the number of orbits  $\ge 3$, then $|R|\le n/3.$ Note that the element  
 $(1,4, \ldots, n-2)(2,5, \ldots , n-1)(3,6, \ldots , n)$ of order $n/3$ stabilizer all three partitions, so it lies in  $R$, from which we obtain 
 $|R|=n/3,$ $ R=\langle (1,4, \ldots, n-2)(2,5, \ldots , n-1)(3,6, \ldots , n) \rangle$.
 
 We have $G \cap G^{(1,2, \ldots, n)} \cap G^{(1,2, \ldots, n)^2} \le R.$ Consider the group 
 $$ G \cap G^{(1,2, \ldots, n)} \cap G^{(1,2, \ldots, n)^2}  \cap G^{(3,4)}.$$ It have to stabilize the partition  $\{1,2,4\}\{3,5,6\} \ldots \{n-2,n-1,n\}$, which,
 as easy to see, is not stabilized by any non-trivial element of  $R$, thus  $$G \cap G^{(1,2, \ldots, n)} \cap G^{(1,2, \ldots, n)^2}  \cap G^{(3,4)}=1.$$
 Therefore, $b_G(S_n)\le4,$ and by Lemma \ref{bs} we obtain that  $Reg_G(S_n,5) \ge 5.$

{\bf Case $k=4$.} Note that the group 
$$((K_1 \times \ldots \times K_l) \rtimes L) \cap ((K_1 \times \ldots \times K_l) \rtimes L)^{(1,2, \ldots, n)^2}$$ 
stabilizes the partition $\{1,2\}\{3,4\}, \ldots , \{n-1,n\}$, which yields
by  the first case that the group 
\begin{align*}
((K_1 \times \ldots \times K_l) \rtimes L) \cap ((K_1 \times \ldots & \times K_l) \rtimes L)^{(1,2, \ldots, n)^2} \cap \\ \cap
(((K_1 \times \ldots & \times K_l) \rtimes L) \cap ((K_1 \times \ldots \times K_l) \rtimes L)^{(1,2, \ldots, n)^2})^{(1,2, 
\ldots, n )}
\end{align*}
 is semiregular.

Let us find the number of orbits of the group
\begin{align*}
R=((K_1 \times \ldots \times K_l) \rtimes S_l) \cap  ((K_1 \times \ldots & \times K_l) \rtimes S_l)^{(1,2, \ldots, n)^2} \cap \\ \cap
(((K_1 \times \ldots  \times K_l) \rtimes S_l) & \cap ((K_1 \times \ldots \times K_l) \rtimes S_l)^{(1,2, \ldots, n)^2})^{(1,2, 
\ldots, n )},
\end{align*}
which stabilizes the partitions  
\begin{gather*}
\{1,2,3,4\}\{5,6,7,8\} \ldots \{n-3,n-2,n-1,n\};\\ \{2,3,4,5\}\{6,7,8,9\} \ldots \{n-2,n-1,n,1\}; \\
\{3,4,5,6\}\{7,8,9,10\} \ldots \{n-1, n,1,2\};\\ \{4,5,6,7\}\{8,9,10,11\} \ldots \{ n,1,2,3\}.
\end{gather*}
Arguments similar to those of the preceding case ($k=3$) shows that
$$R=\langle (1,5, \ldots, n-3)(2,6, \ldots, n-2)(3,7, \ldots, n-1)(4,8, \ldots n) \rangle.$$

We have $G \cap G^{(1,2, \ldots, n)} \cap G^{(1,2, \ldots, n)^2} \cap G^{(1,2, \ldots, n)^3} \le R.$ Consider the group 
 $$ G \cap G^{(1,2, \ldots, n)} \cap G^{(1,2, \ldots, n)^2}  \cap G^{(1,2, \ldots, n)^3} \cap G^{(4,5)}.$$
 It have to stabilize the partition
 $\{1,2,3,5\}\{4,6,7,8\} \ldots \{n-3,n-2,n-1,n\}$, which is not stabilized by the group $ R $, and we obtain that 
 $$G \cap G^{(1,2, \ldots, n)} \cap G^{(1,2, \ldots, n)^2}  \cap G^{(1,2, \ldots, n)^3} \cap G^{(4,5)}=1.$$
 
 Consider the points
 \begin{gather}
 (G, Ga, G{a^2}, G{a^3}, G{(4,5)}); \\
 (G, Ga, G{a^2}, G{a^3}, G{(3,6)}); \\
 (G, Ga, G{a^2}, G{a^3}, G{(2,7)}); \\
 (G, Ga, G{a^2}, G{a^3}, G{(1,8)}); \\
 (G, Ga, G{a^2}, G{a^3}, G{(2,5)}), 
 \end{gather}
 where $a={(1,2, \ldots, n)}$. Arguments similar to those given above, show that all these points are $ G $-regular. We will show that they all lie in different orbits.
 
 Assume that $(G, Ga, G{a^2}, G{a^3}, G{(4,5)})$ and $(G, Ga, G{a^2}, G{a^3}, G{(3,6)})$ lie in the same orbit, i.e. there is
 an element $g \in S_n$, which maps $(G, Ga, G{a^2}, G{a^3}, G{(4,5)})$ to $(G, Ga, G{a^2}, G{a^3}, G{(3,6)})$, then 
 $g, g^{a^{-1}}, g^{a^{-2}}, g^{a^{-3}}, (4,5)g(3,6) \in G$.
It is clear that a trivial permutation does not map any of the points (1) -- (5) to another, i.e., we can assume that $ g \ne1. $
 Then the element  
 $g$ lies in the group
 $$G \cap G^{a} \cap G^{a^{2}} \cap G^{a^{3}}=R,$$ therefore $g=r^m$,
 where 
 \begin{align*}
 r=(1, 5, \ldots, i_m, \ldots, n-3)(2, 6, \ldots, & j_m, \ldots, n-2) \cdot \\ \cdot (3, 7, \ldots, k_m, & \ldots, n-1)(4, 8, \ldots, l_m, \ldots, n),  1\le m <n/4.
 \end{align*}
 The element $(4,5)g(3,6)$ maps 1 to $i_m$, and $4$ to $i_{m+1}$, and therefore does not stabilize the partition $\{1,2,3,4\}\{5,6,7,8\} \ldots \{n-3,n-2,n-1,n\},$
 i.e. can not lie in $ G $, so (1) and (2) lie in distinct orbits. 
The fact that (1) and (3), (1) and (5), (2) and (4), (2) and (5), (3) and (4), (3) and (5 ), (4) and (5) does not
  lie in the same orbit can be proved by similar arguments.

 Assume that $(G, Ga, G{a^2}, G{a^3}, G{(4,5)})$ and $(G, Ga, G{a^2}, G{a^3}, G{(1,8)})$ lie in the same orbit, i.e. there is
 an element $g\in S_n$, which maps $(G, Ga, G{a^2}, G{a^3}, G{(4,5)})$ to $(G, Ga, G{a^2}, G{a^3}, G{(1,8)})$. As above we obtain $g=r^m.$
 The element $(4,5)g(1,8)$ maps 1 to $i_m$. If $i_{m+1} \ne1,$ then, as above,  $(4,5)g(1,8)$ maps 4 to $i_{m+1}$
 and $(4,5)g(1,8)$ can not lie in $G$. Let $i_{m+1}=1$, then $(4,5)g(1,8)$ maps 4 to 8. Then, because  the element $(4,5)g(1,8)$ lies in $G$ and stabilize the corresponding partition, it is necessary that   $i_m=5$. Therefore we obtain that $g=r=r^{-1},$ and it is a contradiction for all cases except the case  $G\le S_4 \wr S_2$, for which the statement of the Lemma is known. The fact that (2) and (3) does not
  lie in the same orbit can be proved by similar arguments.

Therefore, $Reg_G(S_n,5)\ge5$, as wanted.
\end{proof}

\section{Proof of the main theorem}

\begin{Lemma}\label{osn}
Let $G$ be a solvable subgroup of   $S$, where $S$ is isomorphic to   $S_n$ or $A_n$, $n\ge 5$. Then $Reg_G(S,5)\ge 5.$
\end{Lemma}
\begin{proof}
If $G$ is transitive then the lemma follows from Lemma  \ref{tr}. Let $G$ be intransitive, i.e.
$G\le S_{k_1} \times \ldots \times S_{k_l}.$ If  $k_i \ge 5$ for all $i$ then by Lemma \ref{tr} there are
$x_i, y_i, z_i, t_i$, such that $G_i \cap G_i^{x_i} \cap G_i^{y_i} \cap G_i^{z_i} \cap G_i^{t_i}=1$, here $G_i$~is
the projection of the group
$G$ on $S_i,$
hence $$G \cap G^{x_1 \cdot \ldots \cdot x_l} \cap G^{y_1 \cdot \ldots \cdot y_l} \cap G^{z_1 \cdot \ldots \cdot z_l} \cap G^{t_1 \cdot \ldots \cdot t_l}=1,$$
thus $Reg_G(S,5) \ge 5.$

Further proof is by induction on the number of orbits. The base of induction is the cases when there are only the orbits of sizes  2, 3 and 4.
It suffices to consider the cases $S_2 \times S_3$, $S_2 \times S_4$, $S_3 \times S_4$, $S_2 \times S_2 \times S_2$, $S_3 \times S_3$, $S_4 \times S_4$. 
The validity of the statement  $ Reg_G (S, 5) \ge5 $ for these groups can be verified directly with the help of GAP.

Assume that $G\le S_k \times S_{n-k}$, $k=2,3,4; n-k \ge 5.$ Now $S_k$, up to conjugation in $S_n$, acts on  $\{1, \ldots, k\}$, while $S_{n-k}$ acts on $\{k+1, \ldots, n\}.$ Let $G_1$ be the  projection of $G$ on $S_{n-k}$.
As earlier we may assume that the projection of $G$ on $S_k$ equals  $S_k,$ since  $S_k$ is solvable for $k\le 4$.
By induction,  there exist $x,y,z,t \in S_{n-k}$ such that $$G_1 \cap G_1^{x} \cap G_1^{y} \cap G_1^{z} \cap G_1^{t}=1.$$
Denote $K=G_1^{x} \cap G_1^{y} \cap G_1^{z} \cap G_1^{t}$.

{\bf Case $k=2$.} Consider the intersection $S_2 \times G_1 \cap (S_2 \times K)^{(2,3)}$, and let  $g$ lies in the intersection.
Since $g$ lies in $S_2 \times G_1$,  it cannot map   $3$  to $1$ or $2$. On the other hand, as an element of 
$(S_2 \times K)^{(2,3)}$,  $g$ can map   $3$ only to  $1$. So $g$ fixes $3$. Thus it fixes $1$ and hence
it fixes $2$, i.e. $g \in K$.
We obtain that $g$ is contained in
$G_1 \cap K=1.$

Consider elements $x_1, y_1, z_1, t_1 \in S_{n-k}$ such that $G_1 \cap G_1^{x_1} \cap G_1^{y_1} \cap G_1^{z_1} \cap G_1^{t_1}=1$  and points
$(G_1,G_1x,G_1y,G_1z,G_1t)$ and $(G_1,G_1x_1,G_1y_1,G_1z_1,G_1t_1)$ lie in distinct  $S_{n-k}$-orbits. Then,
as we prove above,
\begin{align*}
S_2 \times G_1 \cap (S_2 \times G_1^{x_1})^{(2,3)} \cap & (S_2 \times G_1^{y_1})^{(2,3)} \cap \\ \cap & (S_2 \times  G_1^{z_1})^{(2,3)} \cap (S_2 \times G_1^{t_1})^{(2,3)}=1,
 \end{align*}
and points
\begin{gather*}
(G, Gx(2,3), G{y}{(2,3)}, G{z}{(2,3)}, G{t}(2,3));\\
 (G, Gx_1(2,3), G{y_1}{(2,3)}, G{z_1}{(2,3)}, G{t_1}(2,3))
 \end{gather*}
lie in the same $S_n$-orbit
if and only if there is  $g\in S_n$ such that
$$g,xg^{(2,3)}x_1^{-1}, yg^{(2,3)}y_1^{-1}, zg^{(2,3)}z_1^{-1}, tg^{(2,3)}t_1^{-1} \in G \le S_k \times S_{n-k}.$$
Suppose that such $g$ exists. Then since  $x,y,z,t,x_1,y_1,z_1,t_1 \in S_{n-k}$, we obtain that
$$g \in (S_k \times S_{n-k}) \cap (S_k \times S_{n-k})^{(2,3)}.$$
As shown above, in this case $g \in S_{n-3}$, in particular $g$ map  $(G_1,G_1x,G_1y,G_1z,G_t)$ to $(G_1,G_1x_1,G_1y_1,G_1z_1,G_1t_1)$, but this is impossible, since the points are from distinct orbits.
Hence $Reg_G(S_n,5) \ge Reg_{G_1}(S_{n-k},5)\ge 5.$

{\bf Case $k=3$.} Consider the group
$$S_3 \times G_1 \cap (S_3 \times G_1^{x})^{(1,4)} \cap (S_3 \times G_1^{y})^{(1,4)(2,5)} \cap S_3 \times G_1^{z} \cap S_3 \times G_1^{t}.$$
Let  $g$ be an element of this group.
Then $g$ stabilizes the point 3, because $G$ lies in $S_3 \times G_1 \cap (S_3 \times G_1^{x})^{(1,4)} \cap (S_3 \times G_1^{y})^{(1,4)(2,5)}$. Since
$g \in S_3 \times G_1 $, $g$ can map 2 only to 2 or 1, but $g \in  (S_3 \times G_1^{x})^{(1,4)} $ so   2 can not be moved to 1, hence   $g$ stabilizes 2, thus $g$ also stabilizes 1 and 4. Now since $g \in (S_3 \times G_1^{y})^{(1,4)(2,5)}$, the element $g$ stabilizes  5, and thus lies in $K$. Therefore
$g=1.$

Consider the elements $x_1, y_1, z_1, t_1 \in S_{n-k}$ such that  $G_1 \cap G_1^{x_1} \cap G_1^{y_1} \cap G_1^{z_1} \cap G_1^{t_1}=1$  and points
$(G_1,G_1x,G_1y,G_1z,G_1t)$ и $(G_1,G_1x_1,G_1y_1,G_1z_1,G_1t_1)$ lie in distinct  $S_{n-k}$-orbits. Then, by what was proved above we obtain
\begin{align*}
S_3 \times G_1 \cap (S_4 \times G_1^{x_1})^{(1,4)} \cap & (S_3 \times G_1^{y_1})^{(1,4)(2,5)} \cap \\ \cap & S_3 \times  G_1^{z_1} \cap S_4 \times G_1^{t_1}=1,
 \end{align*}
and points
\begin{gather*}
(G, Gx(1,4), G{y}{(1,4)(2,5)}, G{z}, G{t});\\
 (G, Gx_1(1,4), G{y_1}{(1,4)(2,5)}, G{z_1}, G{t_1})
 \end{gather*}
lie in the same $S_n$-orbit 
if and only if there is $g\in S_n$ such that
$$g,xg^{(1,4)}x_1^{-1}, yg^{(1,4)(2,5)}y_1^{-1}, zgz_1^{-1}, tgt_1^{-1} \in G \le S_k \times S_{n-k}.$$
Suppose that such $g$ exists then we obtain 
$$g \in (S_k \times S_{n-k}) \cap (S_k \times S_{n-k})^{(1,4)} \cap (S_k \times S_{n-k})^{(1,4)(2,5)},$$
since $x,y,z,t,x_1,y_1,z_1,t_1 \in S_{n-k}$.
As it is shown above, in this case $g \in S_{n-5}$, then $g$ maps $(G_1,G_1x,G_1y,G_1z,G_1t)$ to $(G_1,G_1x_1,G_1y_1,G_1z_1,G_1t_1)$, which contradicts the assumption, since the points are from distinct orbits.
Thus $Reg_G(S_n,5) \ge Reg_{G_1}(S_{n-k},5)\ge 5.$

{\bf Case $k=4$.} Consider the group
$$S_4 \times G_1 \cap (S_4 \times G_1^{x})^{(1,5)} \cap (S_4 \times G_1^{y})^{(1,5)(2,6)} \cap (S_4 \times G_1^{z})^{(1,5)(2,6)(3,7)} \cap S_4 \times G_1^{t},$$
Let  $g$ be an element of this group. As easy to see, $g$ stabilize the point 4. Since $g \in S_4 \times G_1 \cap (S_4 \times G_1^{y})^{(1,5)(2,6)}$, it stabilize the point 3. Hence, because  $g \in S_4 \times G_1 \cap (S_4 \times G_1^{x})^{(1,5)}$, we obtain that  2 and 5 also fixed by $g$. Thus $g$ also fixes 1 and 6 since $g \in S_4 \times G_1 \cap (S_4 \times G_1^{y})^{(1,5)(2,6)}$. Therefore $g$ also
stabilizes 7. Thus $g \in K$, i.e. $g=1.$

Consider the elements $x_1, y_1, z_1, t_1 \in S_{n-k}$ such that $G_1 \cap G_1^{x_1} \cap G_1^{y_1} \cap G_1^{z_1} \cap G_1^{t_1}=1$  and points
$(G_1,G_1x,G_1y,G_1z,G_1t)$ and $(G_1,G_1x_1,G_1y_1,G_1z_1,G_1t_1)$ lie in distinct $S_{n-k}$-orbits. Then, by what was proved above, we obtain that
\begin{align*}
S_4 \times G_1 \cap (S_4 \times G_1^{x_1})^{(1,5)} \cap & (S_4 \times G_1^{y_1})^{(1,5)(2,6)} \cap \\ \cap & (S_4 \times  G_1^{z_1})^{(1,5)(2,6)(3,7)} \cap S_4 \times G_1^{t_1}=1,
 \end{align*}
and points
\begin{gather*}
(G, Gx(1,5), G{y}{(1,5)(2,6)}, G{z}{(1,5)(2,6)(3,7)}, G{t});\\
 (G, Gx_1(1,5), G{y_1}{(1,5)(2,6)}, G{z_1}{(1,5)(2,6)(3,7)}, G{t_1})
 \end{gather*}
lie in the same $S_n$-orbit 
if and only if there is $g\in S_n$ such that
$$g,xg^{(1,5)}x_1^{-1}, yg^{(1,5)(2,6)}y_1^{-1}, zg^{(1,5)(2,6)(3,7)}z_1^{-1}, tgt_1^{-1} \in G \le S_k \times S_{n-k}.$$
Suppose that such $g$ exists then we obtain that
$$g \in (S_k \times S_{n-k}) \cap (S_k \times S_{n-k})^{(1,5)} \cap (S_k \times S_{n-k})^{(1,5)(2,6)} \cap (S_k \times S_{n-k})^{(1,5)(2,6)(3,7)}$$
 since $x,y,z,t,x_1,y_1,z_1,t_1 \in S_{n-k}$.
As it is shown above, in this case $g \in S_{n-7}$,  then $g$ maps $(G_1,G_1x,G_1y,G_1z,G_1t)$ to $(G_1,G_1x_1,G_1y_1,G_1z_1,G_1t_1)$, which contradicts the assumption since the points are from distinct orbits.
Thus $Reg_G(S_n,5) \ge Reg_{G_1}(S_{n-k},5)\ge 5.$

Note that in these cases the group $G$ is normalized by $(1,2)$,
from this property and the fact that  $Reg_G(S_n,5)\ge5$ we obtain that
$Reg_G(A_n,5)\ge5$.

\end{proof}

The main theorem for the case of an almost simple group $S$ with socle isomorphic to $A_n$, $n \ge 5$, $n\ne 6$ follows from Lemma \ref{osn}. 
The main theorem for the case of an almost simple group $S$ with socle isomorphic to $A_6 \cong PSL_2(9)$ follows from \cite[Theorem 1.1]{bur}, if subgroup  $G$ does not lie in Aschbacher class 
 $C_1$,and from \cite[Lemma 8]{vd2}, if subgroup  $G$ lies in Aschbacher class  $C_1$.

\end{document}